\documentclass[reqno,11pt]{amsart}
\usepackage{amsmath,amssymb,amscd,amsthm,latexsym,graphics,enumerate,engord,hyperref,txfonts}
\usepackage[all]{xy}
\usepackage{amsthm,amsmath}
\usepackage{xcolor}
\newtheorem{theo}{Theorem}

\usepackage[english]{babel}
\newtheorem{theorem}{Theorem}[section]
\newtheorem{example}[theorem]{Example}
\newtheorem{remark}[theorem]{Remark}

\newtheorem{lemma}[theorem]{Lemma}

\newtheorem{corollary}[theorem]{Corollary}

\def\cat{{\rm{cat}\hskip1pt}}

\def\TC{{\rm{TC}\hskip1pt}}

\begin{document}
	\title[An upper bound for $\TC_0$ of a family of elliptic spaces]{An upper bound for the rational topological complexity of a family of elliptic spaces}
	
	\author{Said Hamoun}
	\author{Youssef Rami}
	\address{My Ismail University of Mekn\`es, Department of Mathematics, B. P. 11 201 Zitoune, Mekn\`es, Morocco.}
	\email{s.hamoun@edu.umi.ac.ma}
	\email{y.rami@umi.ac.ma}
	\author{Lucile Vandembroucq}
	\address{Centro de Matem\'atica, Universidade do Minho, Campus de Gualtar, 4710-057 Braga, Portugal.}
	\email{lucile@math.uminho.pt}
	\begin{abstract} In this work, we show that, for any simply-connected elliptic space $S$ admitting a pure minimal Sullivan model with a differential of constant length, we have $\TC_0(S)\leq 2\cat_0(S)+\chi_{\pi}(S)$ where $\chi_{\pi}(S)$ is the homotopy characteristic. This is a consequence of a structure theorem for this type of models, which is actually our main result.
	\end{abstract}
	
	
	\maketitle
	
\section{Introduction}
Let $S$ be a path-connected topological space. In his work \cite{FM}, M. Farber introduced the notion of topological complexity of $S$ denoted by $\TC(S)$. This is a homotopy invariant defined as the least integer $m$ for which the map $ev_{0,1}: S^{[0,1]}\rightarrow S\times S$, $\lambda \rightarrow (\lambda  (0),\lambda (1))$ admits $ m+1$ local continuous sections $s_i: U_i \rightarrow S ^{[0,1]}$ where $\{U_i\}_{i=0,\cdots,m}$ is a family of open subsets covering $S\times S$. If $S$ is a simply-connected space of finite type and $S_{\!{0}}$ is its rationalization, then the rational topological complexity of $S$, denoted and defined by $\TC_0(S):=\TC(S_{\!{0}})$, provides a lower bound for $\TC(S)$. Through rational homotopy techniques, $\TC_0$ can be expressed in terms of Sullivan models (\cite{C}, \cite{CKV}) in the same spirit as $\cat_0$, the rational Lusternik–Schnirelmann category, was characterized by Félix and Halperin \cite{FH}. Recall that a Sullivan model of $S$ (model for short) is a commutative differential graded algebra $(\Lambda V,d)$ which contains all the information on the rational homotopy type of $S$. In particular, $H^*(S;\mathbb{Q})=H^*(\Lambda V,d)$ and if the model is minimal, that is, $dV\subset \Lambda^{\geq 2} V$, then we have $V\cong \pi_*(S)\otimes \mathbb{Q}$. The standard reference is \cite{FHT}. When there exists an integer $l\geq 2$ such that $dV\subset \Lambda^{l} V$, we say that $d$ is of constant length $l$. In  particular when $l=2$, $(\Lambda V,d)$ is said \textit{coformal}. In this article, we establish the following result which is an improvement of our Theorem B in \cite{HRV}.
\begin{theo} \label{th1.1.2}
Let $S$ be an elliptic space admitting a pure minimal Sullivan model $(\Lambda V,d)$ where $d$ is of constant length. Then
	$$ \TC_0(S)\leq 2\cat_0(S
	) +\chi_{\pi}(S)$$	
where $\chi_{\pi}(S)$ denotes the homotopy characteristic of $S$.
\end{theo}

Recall that $S$ (or equivalently its minimal model $(\Lambda V,d)$) is {\it elliptic} if both $\dim \pi_*(S)\otimes \mathbb{Q}=\dim V$ and $\dim H^*(S; \mathbb{Q})= \dim H^*(\Lambda V,d)$ are finite. The model $(\Lambda V,d)$ is said {\it pure} if $dV^{even}=0$ and $dV^{odd}\subset \Lambda V^{even}$. We also recall that the homotopy characteristic of $S$ is $\chi_{\pi}(S)=\dim \pi_{even}(S)\otimes \mathbb{Q}-\dim \pi_{odd}(S)\otimes \mathbb{Q}=\dim V^{even}-\dim V^{odd}$. When $S$ is elliptic, we always have $\chi_{\pi}(S)\leq 0$, that is, $\dim V^{odd}\geq \dim V^{even}$. Moreover, if $\chi_{\pi}(\Lambda V):=\chi_{\pi}(S)=0$, then the elliptic model $(\Lambda V,d)$ is said an $F_0$-model. Given a pure elliptic model $(\Lambda V,d)$, we will refer as an \textit{$F_0$-basis extension} to a relative Sullivan model of the form $(\Lambda Z,d) \hookrightarrow (\Lambda V,d)$ where $Z$ is a graded subspace of $V$ and the pure model $(\Lambda Z,d)$ is an $F_0$-model. As is known, the existence of such an $F_0$-basis extension can be impossible, see for instance Example \ref{ex} below.

In \cite[Theorem B]{HRV}, we obtained the same upper bound as in Theorem \ref{th1.1.2} assuming that the differential $d$ has constant length and, in addition, that there exists an $F_0$-basis extension $(\Lambda Z,d) \hookrightarrow (\Lambda V,d)$ such that $Z^{even}=V^{even}$. Here, we will see that this latter additional hypothesis can be relaxed. This will follow from the following structure theorem which, in comparison to \cite[Lemma 8]{H}, may have its own interest. 
	\begin{theo} \label{th1.2.2}
	Let $(\Lambda V,d)$ be a pure elliptic minimal model where $d$ is a differential of constant length. Then there exists an $F_0$-basis extension  $$(\Lambda  Z,d)  \hookrightarrow (\Lambda V,d)$$ where $Z^{even}=V^{even}$.
\end{theo}
Note that this means that $(\Lambda V,d)$ is the model of the total space of a fibration over an $F_0$-space with fibre a product of odd-dimensional spheres.

We prove Theorem \ref{th1.1.2} in Section \ref{S22} and derive its applications to rational topological complexity in Section \ref{S32}.
\section{Structure theorem} \label{S22}
In the sequel, we assume that $S$ is a simply-connected CW-complex of finite type admitting a pure minimal Sullivan model $(\Lambda V,d)$.

 We suppose that $\dim V $ is finite and use the notations $X=V^{even}$ and $Y=V^{odd}$. If $\mathcal{B}=\{x_1,\cdots,x_n\}$ is a basis of $X$, then $(\Lambda V,d)$ is elliptic if and only if for any $x_i \in \mathcal{B}$ there exists $N_i \in \mathbb{N}$ such that $[x_i^{N_i}]=0$ in $H^*(\Lambda V,d)$. It is then easy to see that, given a surjective morphism $\varphi :(\Lambda V,d)\rightarrow (\Lambda W,d)$, if $(\Lambda V,d)$ is pure, minimal and elliptic, then so is $(\Lambda W,d)$.

Let $\alpha_1, \cdots,\alpha_p$ be a family of elements in $\Lambda^+X$. The family $\alpha_1, \cdots,\alpha_p$ is said a \textit{regular sequence} in $\Lambda^+X$ if it satisfies the two following conditions:
\begin{itemize}
\item[•] $\alpha_1$ is not a zero divisor in $\Lambda ^+X$
\item[•] For all $i=2,\cdots,p$, $\alpha_i$ is not a zero divisor in $\Lambda ^+X/( \alpha_1,\cdots , \alpha_{i-1} )$ where $( \alpha_1,\cdots , \alpha_{i-1} )$ is the ideal of $\Lambda ^+X$ generated by $\alpha_1,\cdots,\alpha_{i-1}$.
\end{itemize}
Note that, since we are considering $X=V^{even}$, the first condition is automatically satisfied as soon as $\alpha_1\neq 0$.
We recall the following result due to Halperin. 
\begin{theorem}{(\cite[Lemma 8]{H}, see also \cite[Prop 5.4.5]{F})} \label{th2.1.2}
Let $(\Lambda V,d)= (\Lambda (X\oplus Y),d)$ be a pure elliptic Sullivan model. There exists a basis (not necessarily homogeneous) $u_1, \cdots, u_m$ of $Y$ such that $du_1,\cdots, du_n$ is a regular sequence in $\Lambda X$ with $n=\dim X$. 
\end{theorem}
Recall that a pure model $(\Lambda Z,d)$ such that $\dim Z<\infty$ and $\chi_{\pi}(\Lambda Z)=0$ is an $F_0$-model if and only if there exists a (homogeneous) basis $z_1,\dots, z_n$ of $Z^{odd}$ such that $dz_1,\dots, dz_n$ is a regular sequence in $\Lambda Z^{even}$ (\cite[Prop. 32.10]{FHT}). Given a pure elliptic model $(\Lambda V,d)$, the obvious intuition coming from Theorem \ref{th2.1.2} to obtain an $F_0$-basis extension $(\Lambda Z,d)\hookrightarrow (\Lambda V,d)$ with $Z^{even}=V^{even}$, would be to consider $(\Lambda Z,d)=(\Lambda(x_1,\cdots, x_n,u_1,\cdots,u_n),d).$ Unfortunately, since the elements $u_1,\cdots,u_n$ are not necessarily homogeneous, this does not produce in general a well-defined graded differential algebra. We point out that, in the  result above, Halperin used some commutative algebra arguments which do not take in consideration the homogeneity of the elements with respect to the degree. To be clear and to take off any kind of ambiguity about this fact, we consider the following example taken from \cite{AJ}
\begin{example} \label{ex}{\rm
 Let $(\Lambda V,d)=(\Lambda (X\oplus Y),d)=(\Lambda (x_1,x_2,y_1,y_2,y_3 ),d)$ where $|x_1|=6$, $|x_2|=8$, $dy_1=x_1(x_1^4+x_2^3)$, $dy_2=x_2(x_1^4+x_2^3)$ and $dy_3=x_1^3x_2^2$. We will see that there is no $F_0$-basis extension $(\Lambda  Z,d)  \hookrightarrow (\Lambda V,d)$ with $Z^{even}=X=\langle x_1, x_2 \rangle=\mathbb{Q}x_1\oplus \mathbb{Q} x_2$. We note that $|y_1|=29, |y_2|=31$ and $|y_3|=33$. If there were such an extension then $Z^{odd}$ should be a graded subspace of $Y$. This means that we shoud be able to find a (homogeneous) basis \{$u_1,u_2,u_3\}$ of $Y$ such that $Z^{odd}=\langle u_1,u_2 \rangle$. For degree reasons we can suppose that, up to a scalar, $u_1\in \{y_1,y_2,y_3\}$ and $u_2\in \{y_1,y_2,y_3\}\setminus \{u_1\}$. Since $(\Lambda Z,d)$ is an $F_0$-model, $\{du_1,du_2\}$ must be a regular sequence in $\Lambda(x_1,x_2)$. If $u_1=y_1$ then $du_1=x_1(x_1^4+x_2^3)$ is clearly not a zero divisor in $\Lambda(x_1,x_2)$.\\
As shown in the following table which considers the possible values of $u_2$, we can see that $du_2$ is always a zero divisor in the quotient $\Lambda(x_1,x_2)/( du_1)$. 
\medskip
\begin{center}
		\begin{tabular}{|c|c|c|}
			\hline
			$u_2$& $y_2$ & $y_3$\\
			\hline
			In $\Lambda(x_1,x_2)/( du_1)$ & $x_1dy_2=0$ & $(x_1^4+x_2^3)dy_3=0$\\
			\hline
		\end{tabular}
	\end{center}
\medskip

We can then conclude that there is no regular sequence $\{du_1,du_2\}$ where $u_1=y_1$. Similarly, we can verify that if either $u_1=y_2$ or $u_1=y_3$ then we can not find $u_2$ such that $\{du_1,du_2\}$ is a regular sequence in $\Lambda(x_1,x_2)$. Therefore there is no $F_0$-basis extension $(\Lambda Z,d)\hookrightarrow (\Lambda V,d)$ with $Z^{even}=V^{even}$ and any basis $\{u_1,u_2,u_3\}$ provided by Theorem \ref{th2.1.2} is necessarily non-homogeneous. For instance, we can check that $\{u_1=y_3, u_2=y_1+y_2, u_3=y_3\}$ is a basis of $Y$ such that $\{du_1, du_2\}$ is a regular sequence in $\Lambda(x_1,x_2)$ and the element $u_2$ is not homogeneous since $|y_1|\neq|y_2|$.}
\end{example}
Note that the differential in the example above has non-constant length. In this work, we consider $(\Lambda V,d)$ a pure elliptic model and, as stated in Theorem \ref{th1.2.2}, we will prove that there exists an $F_0$-basis extension $(\Lambda Z,d) \hookrightarrow (\Lambda V,d)$ with $Z^{even}=V^{even}$ whenever $d$ is of constant length. In other words, our result ensures the existence of a homogeneous basis in Theorem \ref{th2.1.2} provided that $d$ is of constant length.  

We first set some notations and prove a special case which will be crucial in the proof of the general case.

 Suppose that ${\mathcal B}=\{x_1, \cdots, x_n\}$ is a basis of $X$ satisfying $|x_1|\leq \cdots \leq |x_n|$ and $\{y_1, \cdots , y_m\}$ a basis of $Y$. Let $X_1:= \langle x_k: |x_k|= |x_1| \rangle$ be the vector subspace of $X$ generated by the elements $x_k$ for which $|x_k|=|x_1|$. Similarly, let $Y_1:= \langle y_k:  |y _k| \leq l |x_1| -1 \rangle $ be the vector subspace of $Y$ generated by the elements $y_k$ satisfying $|y_k| \leq l|x_1|-1$. Notice that if $dy_k\neq 0$ then $|y_k|=l|x_1| -1$. For $V_1= X_ 1\oplus Y_1$ we have $dY_1 \subset \Lambda X_1$ and $(\Lambda V_1 ,d)$ is a pure commutative differential graded algebra, called thereafter the first stage of $(\Lambda V,d)$.
\begin{lemma} \label{lm}
Let $(\Lambda V,d)$ be a pure elliptic model where $d$ is a differential of constant length $l$ and let $(\Lambda V_1,d)$ be the first stage of $(\Lambda V,d)$. Then 
\begin{itemize}
\item[(i)] $(\Lambda V_1,d)$ is pure elliptic.
\item[(ii)] There exists an $F_0$-basis extension $(\Lambda E,d) \hookrightarrow  (\Lambda V_1,d)$ with $E^{even}=V_1^{even}$.
\end{itemize}
\end{lemma}
\begin{proof}
(i) Let $x_k \in {\mathcal B}$ such that $|x_k|=|x_1|$. Since $(\Lambda V,d)$ is an elliptic model then there exists $N_k \in \mathbb{N}\setminus\{0\}$ satisfying $x_k^{N_k}=dP_k$ for some $P_k \in \Lambda V$. Furthermore, since $(\Lambda V,d)$ is pure and $d$ is of constant length $l$, $P_k$ can be written as $ \sum \limits _j m_j\cdot y_j$ where $m_j \in \Lambda ^{N_k-l} X$ and $dy_j\neq 0$ for each $j$. As $|x_1|$ is the lowest degree, we have $|m_j| \geq (N_k-l)|x_1|$ and  $|d y_j | \geq l|x_1|$. Since on the first hand $|dP_k|=N_k|x_k|=N_k |x_1|$ and, on the other hand, $|dP_k|=|m_jdy_j|$ for any $j$ we must have $|m_j|= (N_k -l)|x_1|$ and $|dy_j|=l |x_1|$. Therefore $m_j\in \Lambda X_1$ and $ y_j \in Y _1$ for any  $j$ and then $P_k \in \Lambda V_1.$ This shows that $[x_k^{N_k}]=0$ in $H^*(\Lambda V_1,d)$ and consequently $(\Lambda V_1,d)$ is elliptic.\\
		(ii) We consider 
		\begin{equation*}
			\begin{cases}
				R= \langle y_k \in Y_1: dy_k\neq 0 \rangle, \\
				T= \langle y_k \in Y_1: dy_k= 0\rangle. 
			\end{cases}
		\end{equation*}
 We clearly see that
		\begin{itemize}
			\item[•]  $(\Lambda X_1 \otimes \Lambda R,d)$ is pure,
			\item[•] The elements of $R$ are of the same degree.
		\end{itemize} 
		Moreover, $(\Lambda V_1,d)= (\Lambda(X_1\oplus R),d)\otimes (\Lambda T,0)$.
		From $(i)$ we know that $(\Lambda V_1,d)$ is elliptic, hence so is $(\Lambda (X_1\oplus R),d)$. By Theorem \ref{th2.1.2} applied to $(\Lambda (X_1 \oplus R),d)$, there exists a (a priori not necessarily homogeneous) basis  $u_1, \cdots , u_p, \cdots , u_q$ of $R$ such that $du_1, \cdots , du_p$ is a regular sequence in $\Lambda X_1$ and $p=\dim X_1$. Since all the elements of $R$ have the same degree we can assert that this basis is necessarily homogeneous. In other words, we can decompose $R$ as 
		$$ R=R_1\oplus R_2$$
		where $R_1$ and $R_2$ are two vector subspaces of $R$ such that $\dim R_1=\dim X_1$ and $\Lambda(X_1\oplus R_1),d)$ is an $F_0$-model. Setting $E=X_1\oplus R_1$ we obtain an $F_0$-basis extension $(\Lambda E,d)\hookrightarrow (\Lambda V_1,d)$ with $E^{even}=V_1^{even}$.
	\end{proof}
\begin{remark} \label{rk}
If $(\Lambda V,d)$ is an elliptic pure minimal model with $V^{even}$ concentrated in a single degree then, Jessup in \cite[Lemma 3.3]{J} proved that there always exists an $F_0$-basis extension $(\Lambda Z,d)\hookrightarrow (\Lambda V,d)$ with $ V^{even}=Z^{even}$. Our Lemma \ref{lm} above recovers this result in the particular case where $d$ is a differential of constant length.
\end{remark}
We are now ready to prove our structure theorem, namely Theorem \ref{th1.2.2} from the introduction.	
\begin{proof}[Proof of Theorem \ref{th1.2.2}]
We proceed by induction on $n=\dim V^{even}$. For $n=1$ the result is obvious. By induction, we suppose that for any pure elliptic model $(\Lambda V,d)$ with $\dim V^{even} \leq n-1$ and $d$ a differential of constant length $l$, there exists an $F_0$-basis extension $$(\Lambda Z,d) \hookrightarrow (\Lambda V,d)$$
satisfying $Z^{even}= V^{even}$. Let $(\Lambda V,d)=(\Lambda(x_1,\cdots, x_{n}, y_1,\cdots, y_m),d)$ be a pure elliptic model with $d$ a differential of constant length $l$ and $\dim V^{even}=n$. By Lemma \ref{lm} there exists an extension $$(\Lambda E,d) \hookrightarrow (\Lambda V,d)$$ where $(\Lambda E,d)$ is an $F_0$-model and $\dim E>0$. Here, without loss of generality, we may suppose that $(\Lambda E,d)$ has the form $(\Lambda E,d)=(\Lambda (x_1,\cdots, x_p,y_1,\cdots, y_p),d)$ where $p\geq 1$. We now consider the following fibration $$(\Lambda E,d) \rightarrow(\Lambda V,d)\rightarrow (\Lambda W,\bar{d}):=(\Lambda (x_{p+1},\cdots, x_n,y_{p+1},\cdots, y_m, \bar{d}).$$
		As $(\Lambda V,d)\rightarrow (\Lambda W, \bar{d})$ is a surjective morphism and $(\Lambda V,d)$ is a pure elliptic minimal model with differential of constant length $l$, so is $(\Lambda W, \bar{d})$. Since $\dim W^{even}<n$, we next use the induction hypothesis on $(\Lambda W, \bar{d})$ to ensure the existence of an $F_0$-basis extension
		$$ (\Lambda (x_{p+1},\cdots, x_n,u_{p+1},\cdots,u _{n}), \bar{d}) \hookrightarrow (\Lambda W, \bar{d})=(\Lambda(x_{p+1},\cdots, x_n,y_{p+1},\cdots, y_{m}), \bar{d})$$	
		where $\langle u_{p+1},\cdots,u _{n}\rangle$, the vector space generated by $u_{p+1},\cdots,u _{n}$, is a graded subspace of $\langle y_{p+1},\cdots,y _{m}\rangle$. 
		
		  Let $U=\langle y_1,\cdots, y_p,u_{p+1}, \cdots, u_n  \rangle \subset Y$ be the vector subspace of $Y$ generated by $\{y_1,\cdots, y_p,u_{p+1}, \cdots, u_n \}$ and let $(\Lambda Z,d):=(\Lambda(X\oplus U),d)\subset (\Lambda V,d)$. It is clear that we have an extension $(\Lambda Z,d)\hookrightarrow (\Lambda V,d)$ where $Z^{even}=V^{even}$ and $\chi_{\pi}(\Lambda Z)=0$. In order to prove that this is an $F_0$-basis extension it remains to show that $(\Lambda Z,d)$ is elliptic.  Since $(\Lambda E,d)$ is an elliptic subalgebra of  $(\Lambda Z,d)$, we already know that, for $1\leq i\leq p$, there exist $M_i\in \mathbb{N} $ and $\xi_i\in \Lambda Z$ such that $d\xi_i=x_i^{M_i}$.
		We will now see that the same is true for any $i\in \{p+1,\cdots, n\}$. 
		
		Let us fix $i\in\{ p+1,\cdots, n\}$. It follows from the ellipticity and pureness of 
		$$ (\Lambda (x_{p+1},\cdots, x_n,u_{p+1},\cdots,u _{n}), \bar{d})$$ that there exists an integer $N_i \in \mathbb{N}$ satisfying
	   $$\bar{d}(v_i)=x_i^{N_i} ,\quad \text{ for some } \quad v_i \in \Lambda (x_{p+1},\cdots, x_n)\otimes \Lambda ^1(u_{p+1},\cdots,u _{n}).$$ 
	   As $\Lambda (x_{p+1},\cdots, x_n)\otimes \Lambda ^1(u_{p+1},\cdots,u _{n})\subset \Lambda X\otimes \Lambda^1Y$, we may look at $v_i$ as an element of $\Lambda X\otimes \Lambda^1 Y$ so that we have in the algebra $(\Lambda V,d)$ 
		\begin{equation}\label{gamma}
		 dv_i=x_i^{N_i}+\gamma_i\quad \text{ where } \gamma_i \in \Lambda^+(x_1,\cdots, x_p) \otimes \Lambda (x_{p+1},\cdots, x_n).
		 \end{equation}
%
     In what follows we express the element $\gamma_i$  from (\ref{gamma}) as an element of $$\Lambda ^+(x_1,\cdots,x_p)\otimes  \Lambda(x_i)\otimes \Lambda(x_{p+1}, \cdots,\hat{x}_i,\cdots ,x_n).$$ As usual the notation `` $\hat{}$ ''  means that the corresponding component is omitted. Explicitly we write
      $$\gamma_i= \sum \limits_{(K,k)} \alpha^k _Kx_i ^k\cdot x ^K_{\langle \mathbf{p};i\rangle}$$ 
    where $K=(k_{p+1},\cdots,{k}_{i-1}, {k}_{i+1},\cdots ,k_n) \in \mathbb{N}^{n-p-1}$, $k\geq 0$, 
    $$x ^K_{\langle \mathbf{p};i \rangle}= x_{p+1}^{k_{p+1}} \cdot x_{p+2}^{k_{p+2}} \cdots {x} _{i-1}^{k_{i-1}}\cdot   {x} _{i+1}^{k_{i+1}}\cdots x_{n}^{k_{n}}$$
    and $\alpha ^k_K \in \Lambda^{\geq 1}(x_1,\cdots, x_p)$ is the coefficient of the monomial $x_i ^k\cdot x ^K_{\langle \mathbf{p};i\rangle}$.
    In the notation $x ^K_{\langle \mathbf{p};i \rangle}$, the subscript $\langle \mathbf{p};i\rangle$ means that the factors $x_1,\dots,x_p$ and $x_i$ are omitted. Formula (\ref{gamma}) can then be written as follows:
    $$dv_i=x_i^{N_i}+\sum \limits_{(K,k)} \alpha^k _Kx_i ^{k}\cdot x ^K_{\langle \mathbf{p} ;i\rangle}.$$
	Note that, for degree reasons, there are only a finite number of pairs $(K,k)$ for which $\alpha_K^k\neq 0.$
	
	For any integer $m_i\in \mathbb{N}$, we then have $$d(x_i^{m_i}v_i)=x_i^{N_i+m_i}+\sum \limits_{(K,k)} \alpha^k _{K}x_i ^{k+m_i}\cdot x ^K_{\langle \mathbf{p}; i \rangle}.$$
		From this calculation, we will use the following iterative process. In the first step, we consider the elements $\alpha^k_Kx_i ^{k+m_i}\cdot x ^K_{\langle \mathbf{p};i \rangle}$. Assuming that $m_i$ is sufficiently large (here $m_i\geq N_i$), we have
	$$d(\alpha^k _Kx_i ^{k+m_i-N_i}\cdot x ^K_{\langle \mathbf{p}; i \rangle} v_i)= \alpha^k _Kx_i ^{k+m_i}\cdot x ^K_{\langle \mathbf{p};i\rangle} + \sum \limits_{(K',k')} \alpha^k _K\alpha^{k'} _{K'}x_i ^{k+m_i-N_i+k'}\cdot x ^{K+K'}_{\langle \mathbf{p};i \rangle}$$
	where as before $K'\in \mathbb{N}^{n-p-1}$ and $K+K'$ is the usual component by component sum. Therefore
	$$d(x_i^{m_i}v_i-\sum_{(K,k)}\alpha^k _{K}x_i ^{k+m_i-N_i}\cdot x ^K_{\langle \mathbf{p};i \rangle} v_i)=x_i^{N_i+m_i}-\sum \limits_{(K,k)}\sum \limits_{(K',k')} \alpha^k _K\alpha^{k'} _{K'}x_i ^{k+m_i-N_i+k'}\cdot x ^{K+K'}_{\langle \mathbf{p} ;i\rangle}.$$
		Remark that in this first step, we have $\alpha^k _{K}\alpha^{k'} _{K'}\in \Lambda ^{\geq 2}(x_1,\cdots,x_p)$.
		
		As a second step we consider the elements $\alpha^k _{K}\alpha^{k'} _{K'}x_i ^{k+m_i-N_i+k'}\cdot x ^{K+K'}_{\langle \mathbf{p};i \rangle}$. Again, assuming that $m_i$ is sufficiently large (which is possible because there exist only a finite number of relevant sequences $K,K'$), we can do the following second iteration:
		\begin{eqnarray*}
	d\left(\alpha^k _{K}\alpha^{k'} _{K'}x_i ^{k+m_i-2N_i+k'}\cdot x ^{K+K'}_{\langle \mathbf{p} ;i\rangle}v_i\right)&=& \alpha^k _{K}\alpha^{k'} _{K'}x_i ^{k+m_i-N_i+k'}\cdot x ^{K+K'}_{\langle \mathbf{p};i \rangle}v_i\\
			&+&  \sum \limits_{(K'',k'')} \alpha^k _{K}\alpha^{k'} _{K'}\alpha^{k''} _{K''}x_i ^{k+m_i-2N_i+k'+k''}\cdot x ^{K+K'+K''}_{\langle \mathbf{p};i \rangle}. 
		\end{eqnarray*}
		We thus have
\begin{multline*}
	\!\!\!\!d\left(x_i^{m_i}v_i-\!\sum\limits _{(K,k)}\!\alpha^k _{K}x_i ^{k+m_i-N_i}\cdot x ^K_{\langle \mathbf{p};i \rangle} v_i + \!\sum \limits_{(K,k)}\!\sum \limits_{(K',k')}\!\alpha^k _{K}\alpha^{k'} _{K'}x_i ^{k+m_i-2N_i+k'}\cdot x ^{K+K'}_{\langle \mathbf{p};i \rangle}v_i\right)= x_i^{N_i+m_i}\\ +\sum \limits_{(K,k)}\sum \limits_{(K',k')}\sum \limits_{(K'',k'')} \alpha^k _{K}\alpha^{k'} _{K'}\alpha^{k''} _{K''}x_i ^{k+m_i-2N_i+k'+k''}\cdot x ^{K+K'+K''}_{\langle \mathbf{p};i \rangle}.
		\end{multline*}
		Now, in this second iteration we have $\alpha^k _{K}\alpha^{k'} _{K'}\alpha^{k''} _{K''} \in \Lambda ^{\geq 3}(x_1,\cdots, x_p)$ and with $m_i$ sufficiently large we can reiterate the same process as many times as we want. After $s$ iterations, we can reformulate the obtained expression as
		\begin{equation}d\left(x_i^{m_i}v_i+\sum _{(J,j)} \tilde{\alpha}_{J}^j x_i^j\cdot x_{\langle \mathbf{p};i \rangle}^Jv_i \right) = x_i^{N_i+m_i}+\sum_{(H,h)} \tilde{\beta}_H^h x_i^h\cdot x_{\langle \mathbf{p};i\rangle}^H \label{eq}
		\end{equation}
		where $J,H\in \mathbb{N}^{n-p-1}$, $j,h\geq 0$, $\tilde{\alpha}_J^j\in \Lambda(x_1,\cdots, x_p)$, and $\tilde{\beta}_H^h\in \Lambda^{>s}(x_{1},\cdots, x_p)$. Since $(\Lambda E,d)=(\Lambda(x_1,\cdots, x_p,y_1,\cdots,y_p),d)$ is elliptic, we choose $s\geq f$ where $f$ is its formal dimension. We then have $[\tilde{\beta}_H^h]=0$ in $H(\Lambda E,d)$, that is $\tilde{\beta}_H^h=db_H^h$ with $b_H^h\in\Lambda E$. Consequently the equation (\ref{eq}) implies $$x_i^{N_i+m_i}=d\xi_i$$ 
		where $$\xi_i=x_i^{m_i}v_i+\sum _{(J,j)} \tilde{\alpha}_J^j x_i^j\cdot x_{\langle \mathbf{p};i \rangle}^Jv_i-\sum_{(H,h)} b_H^h x_i^h\cdot x_{\langle \mathbf{p};i \rangle}^H\in \Lambda(x_1,\cdots\!, x_n,y_1,\cdots\!, y_p, u_{p+1},\cdots\!,u _{n}).$$ 
	As we can do the process above for any $i\in\{p+1,\dots, n\}$, we conclude that, 
		  for any $i\in\{p+1,\dots, n\}$, there exist $M_i \in \mathbb{N}$ and $\xi_i\in \Lambda Z$ such that $x_i^{M_i}=d\xi_i$, which completes the proof.
	\end{proof}
		 

\section{Application to rational topological complexity} \label{S32}
	We now use Theorem \ref{th1.2.2} to obtain an upper bound for the rational topological complexity of certain elliptic spaces. We will use the notation $\TC(\Lambda V)$ instead of $\TC_0(S)$  and $\cat(\Lambda V)$ instead of $\cat_0(S)$ where  $(\Lambda V,d)$ is a minimal Sullivan model of $S$. With this notations, Theorem \ref{th1.1.2} from the introduction can be written as
	\begin{theorem}[Theorem \ref{th1.1.2}] \label{th2} Let $(\Lambda V,d)$ be a pure elliptic model with $d$ a differential of constant length. Then
		$$\TC(\Lambda V) \leq 2\cat(\Lambda V)+\chi_{\pi}(\Lambda V).$$
	\end{theorem}
	\begin{proof}
		This follows from Theorem \ref{th1.2.2} and \cite[Th. 4.2]{HRV} (which is the version in terms of models of \cite[Th. B]{HRV}).
	\end{proof}
In particular, if $(\Lambda V,d)$ is coformal then we have the following corollary
\begin{corollary} \label{corollary}
Let $(\Lambda V,d)$ be a pure elliptic coformal minimal model. Then $$\TC(\Lambda V) \leq \dim V.$$
\end{corollary}
\begin{proof}
The result follows directly from the previous theorem and the well-known fact due to Félix and Halperin that the rational LS-category of an elliptic coformal minimal model $(\Lambda V,d)$ satisfies $\cat(\Lambda V)=\dim V^{odd}$ \cite{FH}.
\end{proof}
	We may extend the result obtained above to a particular case of non-pure elliptic minimal models. More precisely :
	\begin{theorem}
		Let $(\Lambda V,d)$ be an elliptic minimal model with $d$ a differential of constant length. If there exists an extension $(\Lambda Z,d) \hookrightarrow (\Lambda V,d)$ where $(\Lambda Z,d)$ is a pure elliptic algebra satisfying $Z^{even}=V^{even}$ then
		$$\TC( \Lambda V) \leq 2\cat(\Lambda V)+\chi_{\pi}(\Lambda V).$$
	\end{theorem}
	\begin{proof}
		Under the conditions of the theorem, we can suppose that $\Lambda V=\Lambda (Z\oplus U)$ where $U$ is a vector subspace of $V^{odd}$. Then, by \cite[Cor. 3.4]{HRV}, we have \begin{eqnarray*}
			\TC ( \Lambda V)&=& \TC( \Lambda (Z\oplus U))\\
			&\leq & \TC( \Lambda Z) + \dim U.
		\end{eqnarray*}
		As $(\Lambda Z,d)$ is an elliptic pure minimal model of differential $d$ of constant length, Theorem \ref{th2} yields $\TC(\Lambda Z)\leq 2\cat(\Lambda Z)+\chi_{\pi}(\Lambda Z)$ and consequently 
		\begin{eqnarray*}
			\TC ( \Lambda V)&\leq& 2\cat(\Lambda Z)+\chi_{\pi}(\Lambda Z)+ \dim U.
		\end{eqnarray*}
		On the other hand, from the Lechuga-Murillo formula established in \cite{LM} (see also \cite{L}), we have $\cat(\Lambda Z)=\dim Z^{odd}+2\dim Z^{even}(l-2)$ where $l$ is the length of the differential. As $\chi_{\pi}(\Lambda Z)= \dim Z^{even}-\dim Z^{odd}$ we then have :
		\begin{eqnarray*}
			\TC ( \Lambda V)&\leq& 2\cat(\Lambda Z)+\chi_{\pi}(\Lambda Z)+ \dim U\\
			&\leq& 2(\dim Z^{odd}+\dim Z^{even}(l-2)) +\dim Z^{even}-\dim Z^{odd} + \dim U\\
			&\leq& 2(\dim U+\dim Z^{odd}+\dim Z^{even}(l-2)) +\dim Z^{even}-\dim Z^{odd}\\
			& +& \dim U - 2\dim U.
		\end{eqnarray*}
		Moreover, we have $\dim V^{odd}=\dim Z^{odd}+\dim U$ and $\dim Z^{even}=\dim V^{even}$. We hence have:
		\begin{eqnarray*}
			\TC ( \Lambda V)&\leq& 2(\dim V^{odd}+ \dim V^{even}(l-2)) + \dim V^{even}-\dim Z^{odd}-\dim U\\
			&\leq& 2(\dim V^{odd}+ \dim V^{even}(l-2)) + \dim V^{even}-\dim V^{odd}.	\end{eqnarray*}
		As we have $\cat (\Lambda V)=\dim V^{odd}+ \dim V^{even}(l-2)$ and $\chi_{\pi}( \Lambda V)=\dim V^{even}-\dim V^{odd},$ we finally obtain $\TC( \Lambda V) \leq 2\cat(\Lambda V)+\chi_{\pi}(\Lambda V).$
	\end{proof}
	In particular, if $(\Lambda V,d)$ is an elliptic coformal model, we have the following corollary
	\begin{corollary}
		Let $(\Lambda V,d)$ be an elliptic coformal minimal model such that there exists an extension $(\Lambda Z,d) \hookrightarrow (\Lambda V,d)$ where $(\Lambda Z,d)$ is a pure elliptic algebra satisfying $Z^{even}= V^{even}$. Then
		$$\TC( \Lambda V) \leq \dim V.$$
	\end{corollary}
	\begin{proof}
		In this case, we have $\cat (\Lambda V) =\dim V^{odd}$ (see \cite{FH}) and by the previous theorem
		\begin{eqnarray*}
			\TC ( \Lambda V)&\leq & 2\cat (\Lambda V)+\chi_{\pi} (\Lambda V)\\
			&\leq & 2\dim V^{odd}+ \dim V^{even}- \dim V^{odd}\\
			&\leq & \dim V.
		\end{eqnarray*}
	\end{proof}
\section*{Acknowledgements}
This work has been partially supported by Portuguese Funds through FCT -- Funda\c c\~ao para a Ci\^encia e a Tecnologia, within the projects UIDB/00013/2020 and UIDP/00013/2020. S.H would like to thank the Moroccan center CNRST --Centre National pour la Recherche Scientifique et Technique for providing him with a research scholarship grant number: 7UMI2020.

\end{document}